\documentclass[10pt,leqno]{amsart}

\usepackage{amsmath}
\usepackage{amsfonts}
\usepackage{amssymb}
\usepackage{graphicx}
\usepackage{color}
\usepackage{hyperref}
\usepackage{xcolor}

\newtheorem{theorem}{Theorem}[section]
\newtheorem{thm}[theorem]{Theorem}
\newtheorem*{theorem*}{Theorem}
\newtheorem{lemma}{Lemma}[section]
\newtheorem{corollary}[theorem]{Corollary}
\newtheorem{proposition}{Proposition}[section]

\newtheorem{definition}[theorem]{Definition}

\newtheorem{remark}[theorem]{Remark}

\def \b {\beta}

\def\Ric{\text{Ric}}

\def\a{\alpha}
\def\l{\lambda}

\def\g{\gamma}

\def\p{\partial}

\def\R{\mathbb{R}}

\def\vp{\varphi}

\def\k{\kappa}

\def\L{{\mathcal L}}

\def\w{\omega}

\def\Ric{\operatorname{Ric}}

\def\n{\nabla}

\numberwithin{equation}{section}

\begin{document}


\title[Eigenvalue Estimates on quaternion-K\"ahler Manifolds]{Eigenvalue Estimates on quaternion-K\"ahler Manifolds }

\author{Xiaolong Li}
\address{Department of Mathematics and Statistics, McMaster University, Hamilton, Ontario, L8S 4K1, Canada}
\email{li1304@mcmaster.ca}

\author{Kui Wang} \thanks{The research of the second author is supported by NSFC No.11601359} 
\address{School of Mathematical Sciences, Soochow University, Suzhou, 215006, China}
\email{kuiwang@suda.edu.cn}


\subjclass[2020]{35P15, 53C26}

\keywords{Quarternion-K\"ahler manifold, eigenvalue comparison, modulus of continuity, orthogonal Ricci curvature}

\begin{abstract}
We prove lower bound for the first closed or Neumann nonzero eigenvalue of the Laplacian on a compact quaternion-K\"ahler manifold in terms of dimension, diameter, and scalar curvature lower bound. It is derived as large time implication of the modulus of continuity estimates for solutions of the heat equation.
We also establish lower bound for the first Dirichlet eigenvalue in terms of geometric data, via a Laplace comparison theorem for the distance to the boundary function. 
\end{abstract}

\maketitle

\section{Introduction}
Let $(M^n, g)$ be a compact $n$-dimensional Riemannian manifold (possibly with a smooth nonempty boundary). Denote by $\Delta$ the Laplace-Beltrami operator associated to the metric $g$. It is well-known that the spectrum of $\Delta$ (the Neumann boundary condition is imposed if $\p M$ is non-empty) consists of pure points spectrum that can be arranged in the order 
$$0=\mu_0<\mu_1 \leq \mu_2 \leq \cdots  \to \infty.$$
The study of the first nonzero eigenvalue $\mu_1$ is an important issue in both mathematics and physics. 
In particular, the problem of establishing lower bounds for $\mu_1$ in terms of geometric data of the manifold received considerable attention in the past few decades and a number of results have been obtained by various authors (see for example \cite{Chavel84}\cite{SYbook}\cite{BQ00}\cite{LL10}).  
For instance, a classical result of Lichnerowicz states that $\mu_{1} \geq n\k$
if $M$ is a closed $n$-dimensional Riemannian manifold with $\Ric \geq (n-1)\k >0$. This follows easily by integrating the Bochner formula. The rigidity was observed by Obata \cite{Obata62}, who showed that the equality occurs if and only if $M$ is isometric to the round sphere of radius $1/\sqrt{\k}$. 
In the nonnegative Ricci curvature case, by refining the gradient estimates of Li \cite{Li79} and Li and Yau \cite{LY80}, 
Zhong and Yang \cite{ZY84} proved the sharp lower bound $\mu_1 \geq \frac{\pi^2}{D^2}$, where $D$ denotes the diameter of $M$. Moreover, it was proved by Hang and Wang \cite{HW07} that the equality happens if and only if $M$ is a circle of radius $D/\pi$ or the interval $[-D/2,D/2]$.

Both Lichnerowicz and Zhong-Yang's results are special cases of the following theorem, which provides sharp lower bound for $\mu_1$ depending on dimension, Ricci curvature lower bound, and diameter.

\begin{thm}\label{thm Riem}
Let $(M^n,g)$ be a compact Riemannian manifold (possibly with a smooth convex boundary) with diameter $D$ and $\Ric \geq (n-1)\k$ for $\k \in \R$. 
Let $\mu_1$ be the first nonzero eigenvalue of the Laplacian on $M$ (with Neumann boundary condition if $\p M \neq \emptyset$).
Then
\begin{equation*}
    \mu_1 \geq \bar{\l}_1(n,\k,D),
\end{equation*}
where $\bar{\l}_1(n,\k,D)$ is the first nonzero Neumann eigenvalue of the one-dimensional eigenvalue problem
\begin{equation*}
  \vp''-(n-1)T_\k \vp' =-\l \vp
\end{equation*}
on the interval $[-D/2,D/2]$, and $T_\k$ is defined in \eqref{def T}.
\end{thm}

Theorem \ref{thm Riem} was first proved independently by Kr\"oger \cite{Kroger92} using gradient estimate method and by Chen and Wang \cite{CW94} using stochastic methods. The above explicit statement appeared first in the work of Bakry and Qian \cite{BQ00}, who also extended Theorem \ref{thm Riem} to the the setting of smooth metric measure spaces using gradient estimates. In 2013, Andrews and Clutterbuck  \cite{AC13}  gave a simple proof using modulus of continuity estimates (see also \cite{ZW17} for an elliptic argument based on \cite{AC13} and \cite{Ni13}).
The sharpness can be seen by constructing a sequence of Riemannian manifolds with $\Ric \geq (n-1) \k$, which geometrically collapse to the interval $[-D/2,D/2]$ (see for example \cite{AC13}).

However, there are very few results specifically for K\"ahler manifolds. 
Lichnerowicz \cite{Lichnerowicz58} showed that if $M$ is closed K\"ahler manifold with $\Ric \geq (n-1)\k >0$, then $\mu_1 \geq 2(n-1)\k$. 
Notice that this is a remarkable improvement of his well-known result in the Riemannian case. 
Lichnerowicz's proof makes use of a complex version Bochner formula (see also \cite[Theorem 6.14]{Ballmannbook}). A different proof using harmonic maps was given by Urakawa \cite{Urakawa87}. 
It is only until recently that the authors of the present paper established the following analogue of Theorem \ref{thm Riem} for K\"ahler manifold in \cite{LW21TAMS}.

\begin{thm}\label{thm Kahler}
Let $(M^m, g, J)$ be a compact K\"ahler manifold of complex dimension $m$ and diameter $D$, whose holomorphic sectional curvature is bounded from below by $4\k_1$ and orthogonal Ricci curvature is bounded from below by $2(m-1)\k_2$ for some $\k_1,\k_2 \in \R$.    
Let $\mu_{1}$ be the first nonzero eigenvalue of the Laplacian on $M$ (with Neumann boundary condition if $M$ has a strictly convex boundary). Then 
\begin{equation*}
    \mu_{1} \geq \bar{\mu}_{1}(m,\k_1,\k_2,D),
\end{equation*}
where $\bar{\mu}_{1}(m,\k_1,\k_2,D)$ is the first Neumann eigenvalue of the one-dimensional eigenvalue problem 
\begin{equation*}
    \vp''-\left(2(m-1)T_{\k_2}+T_{4\k_1} \right)\vp' =-\l \vp
\end{equation*}
on $[-D/2,D/2]$, and $T_\k$ is defined in \eqref{def T}.
\end{thm}

Theorem \ref{thm Kahler} provides the first diameter-depending lower bound for $\mu_1$ for K\"ahler manifolds.
Its proof uses the modulus of continuity approach of Andrews and Clutterbuck \cite{AC13}.
The key idea in taking the K\"ahlarity into consideration is that the Ricci curvature can be decomposed as the sum of holomorphic sectional curvature and orthogonal Ricci curvature. The notion of orthogonal Ricci curvature was introduced recently by Ni and Zheng \cite{NZ18} in the study of comparison theorems on K\"ahler manifolds (see also \cite{NZ19} and \cite{NWZ18} for more results on orthogonal Ricci curvature).  
It turns out that holomorphic section curvature and orthogonal Ricci curvature are more suitable conditions for various comparison theorems on K\"ahler manifolds. On one hand, they reflect more on the K\"ahler structure and lead to sharper results than the Ricci curvature. On the other hand, they are weaker than the well-studied bisectional curvature lower bound, under which comparions theorems for K\"ahler manifold were obtained by Li and Wang \cite{LW05} and Tam and Yu \cite{TY12}. See also \cite{Liu14} for some comparison theorem for K\"ahler manifolds with Ricci curvature bounded from blow. 

Now let's turn to the quaternion-K\"ahler situation. When $M$ is a closed quaternion-K\"ahler manifold of quaternionic dimension $m\geq 2$ (i.e., the real dimension is $4m$),  it was shown by Alekseevsky and Marchiafava in \cite{AM95} via a Bochner-type formula for $1$-forms that 
$\mu_1 \geq 8(m+1)\kappa$, provided the scalar curvature is bound from below by $16m(m+2)\kappa >0$. Moreover, they showed that the equality characterize the quaternionic projective space. 
Given the results in Theorem \ref{thm Riem} and Theorem \ref{thm Kahler}, it is natural to ask whether one can prove lower bounds for $\mu_1$ that depends on the diameter for a quaterion-K\"ahler manifold.

The purpose of this paper is to establish analogous lower bounds for the first nonzero eigenvalue of the Laplacian on a compact quaternion-K\"ahler manifold. The main theorem states
\begin{thm}\label{thm eigenvalue QK}
Let $(M^m, g, I, J, K)$ be a compact quaternion-K\"ahler manifold (possibly with a smooth strictly convex boundary) of quaternionic dimension $m\geq 2$ and diameter $D$. Suppose that the scalar curvature of $M$ is bounded from below by $16m(m+2)\k$ for some $\k \in \R$. 
Let $\mu_1$ be the first nonzero eigenvalue of the Laplacian on $M$ (with Neumann boundary condition if $M$ has a smooth strictly convex boundary). Then 
\begin{equation*}
    \mu_1 \geq \bar{\mu}_1(m,\k,D)
\end{equation*}
where $\bar{\mu}_1(m,\k,D)$ is the first nonzero Neumann eigenvalue of the one-dimensional eigenvalue problem
\begin{equation*}\label{}
    \vp''-\left(4(m-1)T_\k +3T_{4\k}\right)\vp' =-\l \vp 
\end{equation*}
on $[-D/2,D/2]$, and $T_\k$ is defined in \eqref{def T}.
\end{thm}

To the best of our knowledge, Theorem \ref{thm eigenvalue QK} provides the first diameter-dependent lower bound for $\mu_1$ on a quaternion-K\"ahler manifold. It also seems to be the first lower bound for the Neumann boundary condition in the quaternion-K\"ahler setting. 
Moreover, the monotonicty of $\bar{\mu}_1(m,\k,D)$ in $D$ implies that when the diameter is small, the lower bound provided in Theorem \ref{thm eigenvalue QK} is better than any diameter-independent lower bound.

The lower bound in Theorem \ref{thm eigenvalue QK} will be derived as large time implication of the modulus of continuity estimates for solutions of the heat equation. This is consistent with the Riemannian case in \cite{AC13} and the K\"ahler case in \cite{LW21TAMS}.
Recall that the modulus of continuity $\w$ of a continuous function $u$ defined on a metric space $(X,d)$ is defined by 
\begin{equation*}
   \w(s):= \sup \left\{ \frac{u(y)-u(x)}{2} : d(x,y)=2s \right\}.
\end{equation*} 
It was first observed by Andrews and Clutterbuck \cite{AC09a, AC09} that the moduli of continuity for solutions to a class of quasilinear isotropic equations are subsolutions of the corresponding one-dimensional equations. 
This has been extended in a number of ways: to the setting of  Riemannian manifolds in \cite{AC13}\cite{Ni13}, to Bakry-Emery manifolds in \cite{AN12}\cite{LW21JGA}, to K\"ahler manifolds in \cite{LW21TAMS},  to viscosity solutions in \cite{Li16}\cite{LW17}\cite{LTW20}, to elliptic equations in \cite{AX19}, and to fully nonlinear equations in \cite{Li20}.
The most important applications of the modulus of continuity estimates are in obtaining lower bounds for the first nonzero eigenvalue (see \cite{AC13}\cite{Andrewssurvey15}\cite{LW19eigenvalue,LW19eigenvalue2,LW21TAMS}) and for proving sharp lower bound for the fundamental gap for convex domains in the Euclidean space (see \cite{AC11}), as well as for convex domains in the sphere (see \cite{SWW19}\cite{HWZ20}\cite{DSW18}).

It does not seem possible to express $\bar{\mu}_1(m,\k,D)$ explicitly in terms of elementary functions, so we provide some explicit lower bounds below, which give a sort of interpolation between the Zhong-Yang and Licherowicz lower bounds. 
It is similar to the one obtained by Shi and Zhang in \cite{SZ07} for the Riemannian case, which says for $\bar{\l}_1(n,\k,D)$ as in Theorem \ref{thm Riem}, it holds that for $\k\geq 0$, 
\begin{equation*}
    \bar{\l}_1(n,\k,D) \geq \sup_{s \in (0,1)}\left( 4s(1-s)\frac{\pi^2}{D^2}+ s(n-1)\k \right).
\end{equation*}
A slight modification of their proof shows that in the K\"ahler setting, it holds that
\begin{proposition}\label{prop 1.1}
Let $\bar{\mu}_1(m,\k_1,\k_2,D)$ be as in Theorem \ref{thm Kahler}. If $\k_1, \k_2 \geq 0$, then 
\begin{equation*}
   \bar{\mu}_1(m,\k_1,\k_2,D) \geq \sup_{s \in (0,1)}\left( 4s(1-s)\frac{\pi^2}{D^2}+ 2s(m-1)\k_2+4s\k_1 \right).
\end{equation*}
\end{proposition}
We are grateful to Dr. Shoo Seto who pointed Proposition \ref{prop 1.1} to us. 
Similarly, in the quaternion-K\"ahler case, we have
\begin{proposition}\label{prop 1.2}
Let $\bar{\mu}_1(m,\k,D)$ be as in Theorem \ref{thm eigenvalue QK}. If $\k\geq 0$, then 
\begin{equation*}
   \bar{\mu}_1(m,\k,D) \geq \sup_{s \in (0,1)}\left( 4s(1-s)\frac{\pi^2}{D^2}+ 4s(m+2)\k \right).
\end{equation*}
\end{proposition}

In Section 5, we also establish lower bound for the first Dirichlet eigenvalue in terms of geometric data (see Theorem \ref{Thm Dirichlet}).
This is done via a Laplace comparison theorem for the distance to the boundary function on quaternion-K\"ahler manifold. 
Other sections are organized as follows. In Section 2, we review some basic properties of quaternion-K\"ahler manifolds. In Section 3, we derive the modulus of continuity estimates for solutions of quasilinear parabolic equations on a quaternion-K\"ahler manifold. As an application, we prove Theorem \ref{thm eigenvalue QK}. The explicit lower bounds in Proposition \ref{prop 1.1} and \ref{prop 1.2} are proved in Section 4. 

Throughout the paper, we use the following notations.  
The function $T_\k$ is defined for $\k \in \R$ by
\begin{equation}\label{def T}
     T_\k(t)=\begin{cases}
   \sqrt{\k} \tan{(\sqrt{\k}t)}, & \k>0, \\
   0, & \k=0, \\
   -\sqrt{-\k}\tanh{(\sqrt{-\k}t)}, & \k<0.
    \end{cases}
\end{equation}
The function $c_\k$ is defined for $\k \in \R$ by
\begin{equation}\label{def c kappa}
    c_\k(t) =\begin{cases} 
    \cos \sqrt{\k} t & \text{ if } \k >0, \\
    1 & \text{ if } \k =0,\\
    \cosh \sqrt{-\k }t & \text{ if } \k <0.
    \end{cases}
\end{equation}
\section{Quaternion-K\"ahler Manifolds}

In this section, we recall some basic properties about quaternion-K\"ahler manifolds that will be needed in the sequel.
These are proved by Berger \cite{Berger66} and Ishibara \cite{Ishihara74} (see also \cite{Besse87}). 
We shall follow the presentation of \cite{KLZ08} here.

\begin{definition}
A quaternion-K\"ahler manifold $(M^m,g)$ of quaternionic dimension $m$ (the real dimension is $4m$) is a Riemannian manifold with a rank three vector bundle $V \subset End(TM)$ satisfying
\begin{enumerate}
    \item In any coordinate neighborhood $U$ of $M$, there exists a local basis $\{I,J,K\}$ of $V$  such that 
    \begin{eqnarray*}
    && I^2=J^2=K^2=1,\\
    && IJ=-JI=K, \\
    && JK=-KJ=I, \\
    && KI=-IK=J,
    \end{eqnarray*}
    and  for all  $X,Y\in TM$
    \begin{equation*}
        \langle X,Y \rangle =\langle IX,IY \rangle =\langle JX,JY \rangle =\langle KX,KY \rangle
    \end{equation*}
    \item If $\phi \in \Gamma(V)$, then $\n_X \phi \in \Gamma(V)$ for all $X\in TM$. 
\end{enumerate}
\end{definition}
It is worth noting that then the tensors $I, J, K$ may not be globally defined on $M$. For instance, the canonical quaternionic projective space admits no almost complex structure for topological reasons. However, the space spanned by $I, J, K$ may always be defined globally according to the definition. 

A well known fact is that a $4m$-dimensional Riemannian manifold is quaternion-K\"ahler if and only if its restricted holonomy group is contained in $Sp(n)Sp(1)$.
Since the $4$-dimensional Riemannian manifolds with holonomy $Sp(1)Sp(1)$ are simply the oriented Riemanian manifolds, we shall only consider the case $m \geq  2$.

In this paper, we are mostly concerned about the curvature properties of quaternion-K\"ahler manifolds. 
The Riemannian curvature tensor of $(M,g)$ is defined by 
$$R(X,Y,Z,W)=\langle \n_Y\n_X Z -\n_X \n_Y Z+\n_{[X,Y]}Z, W\rangle.$$

\begin{definition}
Let $(M,g)$ be a quaternion-K\"ahler manifold. 
\begin{enumerate}
    \item The quaternionic sectional curvature of $M$ is defined as 
    \begin{equation*}
    Q(X):=\frac{R(X,IX,X,IX) +R(X,JX,X,JX)+R(X,KX,X,KX)}{|X|^2}.
\end{equation*}
\item The orthogonal Ricci curvature of $M$ is defined as 
\begin{equation*}
    \Ric^\perp(X,X):=\Ric(X,X)-Q(X)
\end{equation*}
\end{enumerate}
\end{definition}

First of all, all quaternion-K\"ahler manifolds of quaternionic dimension $m\geq 2$ are Einstein (see for instance \cite[Theorem 1.2]{KLZ08}), namely there exists a constant $\kappa$ such that  $$\Ric=4(m+2)\k.$$ Moreover, we have the following proposition.  
\begin{proposition}\label{prop QK}
Let $(M^m,g)$ be a quaternion-K\"ahler manifold of quaternionic dimension $m\geq 2$ with $\Ric =4(m+2)\k$ for $\k \in \R$. Then 
\begin{enumerate}
    \item $M$ has constant quaternionic sectional curvature, i.e., 
    \begin{equation*}
        Q(X)=12 \kappa |X|^2. 
    \end{equation*}
    \item $M$ has constant orthogonal Ricci curvature, i.e., 
    \begin{equation*}
        \Ric^\perp (X,X)=4(m-1)\kappa |X|^2.
    \end{equation*}
\end{enumerate}
\end{proposition}
\begin{proof}
For (1), see Theorem 1.3 in \cite{KLZ08}. (2) follows directly from the definition  $\Ric^\perp(X,X) =\Ric(X,X) -Q(X)$.
\end{proof}

We end this section with the following useful lemma. 
\begin{lemma}
Let $(M^m,g)$ be a quaternion-K\"ahler manifold of quaternionic dimension $m\geq 2$ with $\Ric =4(m+2)\k$ for $\k \in \R$. Let $\g : [a,b] \to M$ be a geodesic with unit speed and $X_I(t)$,$X_J(t)$,$X_K(t)$ are parallel vector fields along $\gamma$ such that $X_I(a) = I\gamma'(a)$, $X_J(a) = J\gamma'(a)$, $X_K(a) = K\gamma'(a)$. Then
    \begin{equation*}
        R(\gamma',X_I(t),\gamma',X_I(t))+R(\gamma',X_J(t),\gamma',X_J(t))+ R(\gamma',X_K(t),\gamma',X_K(t)) =12\k
    \end{equation*}
    for all $t \in [a,b]$.
\end{lemma}
\begin{proof}
See Lemma 1.5 in \cite{KLZ08}.
\end{proof}

\section{Modulus of Continuity Estimates}

In this section, we prove the modulus of continuity estimates for solution of a class of quasilinear isotropic parabolic equations on a compact quaternion-K\"ahler manifold.

As in the Riemannian case in \cite{AC13} or the K\"ahler case in \cite{LW21TAMS}, we consider the following isotropic quasilinear equations: 
\begin{equation}\label{eq isotropic}
\frac{\p u}{\p t} =Q[u]:=\left[\a(|
\n u|) \frac{\n_iu \n_j u}{|\n u|^2} +\b(|\n u|)\left(\delta_{ij}-\frac{\n_i u \n_j u}{|\n u|^2}\right) \right].   
\end{equation}
Here $\a$ and $\b$ are smooth positive functions. Some important examples of \eqref{eq isotropic} are the heat equation (with $\a=\b=1$), the $p$-Laplacian heat flows (with $\a=(p-1)|\n u|^{p-2}$ and $\b=|\n u|^{p-2}$) and the graphical mean curvature flow (with $\a=1/(1+|\n u|^2)$ and $\b=1$). 

In the quaternion-K\"ahler case, the associated one-dimensional operator  $\mathcal{F}$ is given by 
\begin{equation}\label{def F}
    \mathcal{F} \vp := \a(\vp')\vp'' +\left(4(m-1)T_{\k}+3T_{4\k} \right)\b(\vp')\vp',
\end{equation}
where the function $T_{\k}$ is the function defined in \eqref{def T}. 

The main result of this section is the following modulus of continuity estimates  on a compact quaternion-K\"ahler manifold, in terms of initial oscillation, elapsed time, and scalar curvature lower bound. 
\begin{thm}\label{thm MC QK}
Let $(M^m,g,I,J,K)$ be a compact quaternion-K\"ahler manifold with diameter $D$ whose scalar curvature is bounded from below by $16m(m+2)\k$ for $\k \in \R$.  
Let $u : M \times [0,T) \to \R$ be a solution of \eqref{eq isotropic} (with Neumann boundary condition if $M$ has a strictly convex boundary). Then the modulus of continuity $\omega:[0,D/2] \times [0,T) \to \R$ of $u$ is viscosity subsolution of the one-dimensional equation
\begin{equation}\label{ODE}
    \w_t =\mathcal{F}\w,
\end{equation}
where the operator $\mathcal{F}$ is defined in \eqref{def F}. 
\end{thm}

\begin{proof}
The proof proceeds as in the K\"ahler case in \cite{LW21TAMS} but is slightly more involved. 
By the definition of viscosity solutions (see \cite{CIL92}), we need to show that for every smooth function $\vp$ that touches $\w$ from above at $(s_0,t_0) \in (0,D/2)\times (0,T)$ in the sense that 
\begin{equation*}
    \begin{cases}
    \vp (s,t)  \geq  \w(s,t)  \text{  near } (s_0,t_0),\\ 
    \vp(s_0,t_0) =  \w(s_0,t_0), 
    \end{cases}
\end{equation*}
it holds that  
\begin{equation}\label{eq2.0}
\vp_t \leq \L  \vp, 
\end{equation}
at the point $(s_0,t_0)$. 
It follows from the definition of $\w$ that for such a function $\vp$, we have
\begin{equation} \label{eq2.1}
    u\left(\gamma(b),t \right)- u\left( \gamma(a),t \right) -2 \vp \left( \frac{L[\gamma]}{2}, t \right) \leq 0
\end{equation}
for any $t \leq t_0$ close to $t_0$ and any smooth path $\gamma: [a,b] \to M$ with length close to $2s_0$. Moreover, since $M$ is compact, there exist points $x_0$ and $y_0$ in $M$ (assume for a moment that $\p M = \emptyset$), with $d(x_0,y_0)=2s_0$ such that the equality in \eqref{eq2.1} holds for $\gamma_0:[-s_0,s_0] \to M$, a length-minimizing unit speed geodesic connecting $x_0$ and $y_0$.
The key idea is to derive useful inequalities from the first and second tests along smooth family of variations of the curve $\gamma_0$. 
For this purpose, we need to recall the first and second variation formulas of arc length.  If $\gamma:(r,s) \to M$ is a smooth variation of $\gamma_0(s)$, then we have 
\begin{equation*}\label{eq 1st variation}
   \left. \frac{d}{dr}\right|_{r=0}L[\gamma(r,s)] = \left.g(T, \gamma_r) \right|_{-s_0}^{s_0}, 
\end{equation*}
and 
\begin{equation*}\label{eq 2nd variation}
    \left. \frac{d^2 }{dr^2}\right|_{r=0}L[\gamma(r,s)]  = \int_{-s_0}^{s_0} \left(|(\n_s \gamma_r)^\perp|^2 -R(\gamma_s,\gamma_r,\gamma_s,\gamma_r)\right)ds +\left. g(T,\n_r \gamma_r)\right|_{-s_0}^{s_0},
\end{equation*}
where $T$ is the unit tangent vector to $\gamma_0$. 

It will be convenient to work in the (quaternionic) Fermi coordinates along $\gamma_0$ chosen as follows. 
Choose an orthonormal basis $\{e_i\}_{i=1}^{4m}$ for $T_{x_0}M$ with
\begin{equation*}
    e_1=\gamma_0'(-s_0), e_2=I\gamma_0'(-s_0), e_3=J\gamma_0'(-s_0), e_4=K\gamma_0'(-s_0),
\end{equation*}
and parallel transport it along $\gamma_0$ to produce an orthonormal basis $\{e_i(s)\}_{i=1}^{4m}$ for $T_{\gamma_0(s)}M$ with $e_1(s)=\gamma_0'(s)$ for each $s\in [-s_0,s_0]$. Notice that the vector fields $I\gamma_0'(t)$, $J\gamma_0'(t)$ and $K\gamma_0'(t)$ may not be parallel along $\gamma_0$. 

The first variation consideration gives 
\begin{equation}\label{eq6.2}
    Q[u](y_0,t_0) -Q[u](x_0,t_0) -2\vp_t \geq 0. 
\end{equation}
and 
\begin{equation}\label{eq6.3}
    \n u(x_0,t_0) =-\vp'e_1(-s_0), \text{ and }
    \n u(y_0,t_0) =\vp'e_1(s_0).
\end{equation}
The variation  $\gamma(r,s)=\gamma_0\left(s+r\frac{2s-1}{L} \right)$ yields 
\begin{equation}\label{eq6.6}
    u_{11}(y_0,t_0)-u_{11}(x_0,t_0) - 2\vp'' \leq 0.
\end{equation}
For $i=2,3,4$, the variations $\gamma(r,s)=\exp_{\gamma_0(s)}\left(r\eta(s)e_i(s) \right)$ produce 
\begin{equation}\label{eq6.7}
    u_{ii}(y_0,t_0) -u_{ii}(x_0,t_0) - \vp' \int_{-s_0}^{s_0} \left((\eta')^2 - \eta^2 R(e_1, e_i,e_1,e_i) \right)ds \leq 0,
\end{equation}
and for $5\leq j \leq 4m$, the variations $\gamma(r,s)=\exp_{\gamma_0(s)}\left(r\zeta(s)e_j(s) \right)$ yield
\begin{equation}\label{eq6.8}
     u_{jj}(y_0,t_0)-u_{jj}(x_0,t_0) - \vp'  \int_{-s_0}^{s_0} \left((\zeta')^2 - \zeta^2 R(e_1, e_j,e_1,e_j) \right)ds \leq 0.
\end{equation}
We derive, by choosing $\eta(s)=\frac{c_{4\k}(s)}{c_{4\k}(s_0)}$ and summing \eqref{eq6.7} for $i=2,3,4$,  that  
\begin{eqnarray}\label{eq6.9}\nonumber
&& \sum_{i=2}^4 \left(u_{ii}(y)-u_{ii}(x) \right) \nonumber \\ \nonumber
&\leq&  \vp'  \int_{-s_0}^{s_0} \left(3(\eta')^2-\eta^2 \left( \sum_{i=2}^4 R(e_1,e_i,e_1,e_i) \right) \right) ds \\ \nonumber
&\leq& 3\vp' \left. \eta \eta' \right|_{-s_0}^{s_0} -\int_{-s_0}^{s_0} \eta^2 \left(\sum_{i=2}^4 R(e_1,e_i,e_1,e_i)-12\k \right)ds \\
&=& -6 T_{4\k} \vp' 
\end{eqnarray}
where we have used 
$\sum_{i=2}^4 R(e_1,e_i,e_1,e_i) \geq 12\k$ by in Proposition \ref{prop QK}. 

Note that 
$$\sum_{i=5}^{4m} R(e_1,e_i,e_1,e_i) =\Ric(e_1,e_1)- \sum_{i=2}^{4} R(e_1,e_i,e_1,e_i) \geq 4(m-1)\k.$$
Choosing $\zeta(s)=\frac{c_{\k}(s)}{c_{\k}(s_0)}$ and summing \eqref{eq6.8} for $5\leq j\leq 4m$ yield
\begin{eqnarray}\label{eq6.10}\nonumber
&& \sum_{j=5}^{4m} \left(u_{jj}(y)-u_{jj}(x) \right)\\ \nonumber
&\leq&  \vp' \int_{-s_0}^{s_0} \left(4(m-1)(\eta')^2-\zeta^2 \left( \sum_{j=5}^{4m} R(e_1,e_j,e_1,e_j) \right) \right) ds \\ \nonumber
&\leq& 4(m-1) \vp' \left. \zeta \zeta' \right|_{-s_0}^{s_0} -\int_{-s_0}^{s_0} \zeta^2 \left(\sum_{j=5}^{4m} R(e_1,e_j,e_1,e_j)-4(m-1)\k \right)ds \\
&=& -8 T_{4\k} \vp'. \nonumber
\end{eqnarray}
Combining \eqref{eq6.9} and \eqref{eq6.10} together, we have 
\begin{eqnarray} \label{eq6.11}\nonumber
&& Q[u](y_0,t_0)-Q[u](x_0,t_0) \\ \nonumber
&=&   \a(\vp')\left(u_{11}(y_0,t_0)-u_{11}(x_0,t_0) \right) +\b(\vp')\sum_{i=2}^{4m}\left(u_{ii}(y_0,t_0) -u_{ii}(x_0,t_0) \right)\\ \nonumber
&\leq& 2\a(\vp')\vp'' -2\b(\vp') \vp' \left( 4(m-1)T_{\k}+3T_{4\k} \right) = 2\mathcal{F}\vp. \\
&&
\end{eqnarray}
It follows from \eqref{eq6.2} and \eqref{eq6.11} that
\begin{equation*}
    \vp_t \leq \mathcal{F}\vp. 
\end{equation*}
This completes the proof. 

In case $M$ has a strictly convex boundary, the same argument as in \cite{LW21TAMS} rule out the possibility that either $x_0\in \p M$ or $y_0 \in \p M$. So the above argument remains valid. 
\end{proof}

The following corollary is immediate. 
\begin{corollary}\label{corollary QK}
Let $M$ and $u$ be the same as in Theorem \ref{thm MC QK}. Suppose $\vp: [0,D/2]\times [0,T) \to \R$ satisfies 
\begin{enumerate}
    \item $\vp_t \geq \mathcal{F}\vp$;
    \item $\vp' \geq 0$ on $[0,D/2]\times [0,T)$;
    \item $|u(y,0)-u(x,0)| \leq 2\vp \left(\frac{d(x,y)}{2},0 \right)$
\end{enumerate}
Then 
\begin{equation*}
    |u(y,t)-u(x,t)| \leq 2\vp \left(\frac{d(x,y)}{2},t \right)
\end{equation*}
for all $x,y\in M$ and $t\in [0,T)$. 
\end{corollary}


As in the Riemannian or K\"ahler case, the modulus of continuity estimate implies lower bounds for the first nonzero eigenvalue of the Laplacian on a quaternion-K\"ahler manifold. We restate Theorem \ref{thm eigenvalue QK} here. 
\begin{thm}\label{thm eigenvalue QK 2}
Let $(M^m, g, I, J, K)$ be a compact quaternion-K\"ahler manifold (possibly with a smooth strictly convex boundary) of quaternionic dimension $m$ and diameter $D$. Suppose that the scalar curvature of $M$ is bounded from below by $16m(m+2)\k$ for some $\k \in \R$. 
Let $\mu_1$ be the first nonzero eigenvalue of the Laplacian on $M$ (with Neumann boundary condition if $\p M \neq \emptyset$). Then 
\begin{equation*}
    \mu_1 \geq \bar{\mu}_1(m,\k,D)
\end{equation*}
where $\bar{\mu}_1(m,\k,D)$ is the first nonzero Neumann eigenvalue of the one-dimensional eigenvalue problem
\begin{equation*}\label{}
    \vp''-\left(4(m-1)T_\k +3T_{4\k}\right)\vp' =-\l \vp 
\end{equation*}
on $[-D/2,D/2]$, and $T_\k$ is defined in \eqref{def T}.
\end{thm}

\begin{proof}[Proof of Theorem \ref{thm eigenvalue QK 2}]
The proof is a slight modification of the proof in the K\"ahler case in \cite{LW21TAMS}, so we leave the details to interested reader. 
\end{proof}

\section{Explicit Lower Bounds}

Let $a,b$ be positive integers and $\k_1,\k_2$ be nonnegative real numbers. 
Consider the one-dimensional eigenvalue problem
\begin{equation}\label{eq:5.1}
    \vp''-\left[2a(m-1)\sqrt{\k_2}\tan(\sqrt{\k_2}t)+2b\sqrt{\k_1} \tan(2\sqrt{\k_1}) \right]\vp'+\l \vp =0
\end{equation}
with Neumann boundary condition $\vp'(-D/2)=\vp'(D/2)=0$. 
\begin{proposition}\label{prop 5.1}
Let $\l_1$ be the first nonzero Neumann eigenvalue of \eqref{eq:5.1}. Then 
\begin{equation*}
     \l_1 \geq \sup_{s \in (0,1)}\left( 4s(1-s)\frac{\pi^2}{D^2}+ s(2a(m-1)\k_2+4b\k_1)\right).
\end{equation*}
\end{proposition}
\begin{proof}
Let $\vp$ be the eigenfunction of \eqref{eq:5.1} associated to the first nonzero Neumann eigenvalue $\l_1$. 
It's easy to observe that the function $y=\vp'$ satisfies the ODE
\begin{eqnarray*}
&& y''-\left[2a(m-1)\sqrt{\k_2}\tan(\sqrt{\k_2}t)+2b\sqrt{\k_1} \tan(2\sqrt{\k_1}) \right]y'  \\ 
&& -\left[2a(m-1)\k_1\sec^2(\sqrt{\k_2}t)+4b\k_1 \sec^2(2\sqrt{\k_1}) \right]y +\l y =0.
\end{eqnarray*}
with Dirichlet boundary condition $y(-D/2)=y(D/2)=0$. 
Multiplying by $y^{\g-1}$ for some $\g>1$ and integrating both sides, we have 
\begin{eqnarray}\label{eq:5.2}
&& \int_{-\frac{D}{2}}^{\frac{D}{2}} y''y^{\g-1} dt +\l \int_{-\frac{D}{2}}^{\frac{D}{2}} y^\g dt\\ \nonumber
&=& \int_{-\frac{D}{2}}^{\frac{D}{2}} \left(2a(m-1)\sqrt{\k_2}  \tan(\sqrt{\k_2}t) +2b\sqrt{k_1}\tan(2\sqrt{\k_1}t)  \right)y' y^{\g-1} dt \\ \nonumber
&& +\int_{-\frac{D}{2}}^{\frac{D}{2}} \left(2a(m-1)\k_2\sec^2(\sqrt{\k_2}t)+4b\k_1 \sec^2(2\sqrt{\k_1})  \right)y^{\g} dt
\end{eqnarray}
By integration by parts, we get
\begin{eqnarray}\label{eq:5.4}
&&\sqrt{\k_2} \int_{-\frac{D}{2}}^{\frac{D}{2}} \tan(\sqrt{\k_2t})y^{\g-1}y' dt \\
   & =&\frac{\sqrt{\k_2}}{\g} \int_{-\frac{D}{2}}^{\frac{D}{2}} \tan(\sqrt{\k_2t})(y^{\g})' dt \nonumber\\
    &=&-\frac{\k_2}{\g} \int_{-\frac{D}{2}}^{\frac{D}{2}}\sec^2(\sqrt{\k_2}t) y^{\g} dt.\nonumber
\end{eqnarray}
Similarly, we have
that 
\begin{equation}\label{eq:5.5}
    2\sqrt{\k_1}\int_{-\frac{D}{2}}^{\frac{D}{2}}\tan(2\sqrt{k_1} t) y' y^{\g-1} dt =-\frac{4\k_1}{\g} \int_{-\frac{D}{2}}^{\frac{D}{2}} \sec^2(2\sqrt{k_1}t) y^{\g} dt 
\end{equation}

On the other hand, integration by parts implies
\begin{equation*}
    \int_{-\frac{D}{2}}^{\frac{D}{2}} y^{\g-1} y'' dt =-(\g-1) \int_{-\frac{D}{2}}^{\frac{D}{2}} y^{\g-2}y'y'\, dt=-\frac{4(\g-1)}{\g^2}\int_{-\frac{D}{2}}^{\frac{D}{2}} \left((y^{\frac{\g}{2}})'\right)^2 dt.
\end{equation*}
By Wirtinger's inequality, 
we get that 
\begin{equation}\label{eq:5.6}
    \int_{-\frac{D}{2}}^{\frac{D}{2}} y^{\g-1} y'' dt \leq -\frac{4(\g-1)}{\g^2} \frac{\pi^2}{D^2} \int_{-\frac{D}{2}}^{\frac{D}{2}} y^\g dt
\end{equation} 

Plugging \eqref{eq:5.4}, \eqref{eq:5.5}, and \eqref{eq:5.6} into \eqref{eq:5.2}, we obtain 
\begin{eqnarray*}
&&-\frac{4(\g-1)}{\g^2} \frac{\pi^2}{D^2} \int_{-\frac{D}{2}}^{\frac{D}{2}} y^\g dt + \l \int_{-\frac{D}{2}}^{\frac{D}{2}}  y^\g\, dt \\
&\geq& \left(1-\frac{1}{\g} \right)\int_{-\frac{D}{2}}^{\frac{D}{2}} \left(2a(m-1)\k_2\sec^2(\sqrt{\k_2}t)+4b\k_1 \sec^2(2\sqrt{\k_1}t)  \right)y^{\g} dt \\
&\geq& \left(1-\frac{1}{\g}\right)\left(2a(m-1)\k_2+4b\k_1   \right)
\int_{-\frac{D}{2}}^{\frac{D}{2}} y^{\g} dt
\end{eqnarray*}
where we have used $\sec^2(x)\geq 1$ for all $x$ in the last line.
It then follows that we must have 
\begin{equation*}
    \l \geq \frac{4(\g-1)}{\g^2} \frac{\pi^2}{D^2} + \left(1-\frac{1}{\g} \right) (2a(m-1)\k_2+4b\k_1 )
\end{equation*}
Letting $s=1-\frac{1}{\g}$, we have
\begin{equation*}
    \l\geq 4s(1-s)\frac{\pi^2}{D^2}+ s(2a(m-1)\k_2+4b\k_1).
\end{equation*}
Since $\g>1$ is arbitrary, we get 
\begin{equation*}
    \l \geq \sup_{s \in (0,1)}\left( 4s(1-s)\frac{\pi^2}{D^2}+ s(2a(m-1)\k_2+4b\k_1)\right).
\end{equation*}
\end{proof}

\begin{proof}[Proof of Proposition \ref{prop 1.1} and \ref{prop 1.2}]
Proposition \ref{prop 1.1} is a special case of Proposition \ref{prop 5.1} with $a=b=1$ and Proposition \ref{prop 1.2} is a special case of Proposition \ref{prop 5.1} with $a=2$ and $b=3$. 
\end{proof}

\section{First Dirichlet Eigenvalue}
Throughout this section, $d(x,\p M)$ denotes the distance function to $\p M$ given by 
\begin{equation*}
    d(x,\p M)=\inf\{d(x,y) : y \in \p M \}
\end{equation*}
and $R$ denotes the inradius of $M$ given by
\begin{equation*}
    R=\sup \{d(x,\p M) :x \in  M \}.
\end{equation*}
For convenience, denote by $C_{\kappa, \Lambda}(t)$ the unique solution of the initial value problem
\begin{equation}\label{C def}
    \begin{cases} 
    \phi''+\kappa \phi =0, \\
    \phi(0)=1,     \phi'(0) =-\Lambda,
    \end{cases}
\end{equation}
and define $T_{\kappa,\Lambda}$ for $\kappa, \Lambda \in \R$ by 
\begin{equation}\label{def T kappa Lambda}
    T_{\kappa,\Lambda} (t):=- \frac{C'_{\kappa, \Lambda}(t)}{C_{\kappa, \Lambda}(t)}.
\end{equation}

In the Riemannian setting, Li and Yau \cite{LY80} and Kause \cite{Kasue84} proved the following well-known result. 
\begin{thm}\label{Thm B}
Let $(M^n,g)$ be a compact Riemannian manifold with smooth boundary $\p M \neq \emptyset$. 
Suppose that the Ricci curvature of $M$ is bounded from below by $(n-1)\kappa$ and the mean curvature of $\p M$ is bounded from below by $(n-1)\Lambda$ for some $\kappa, \Lambda \in \R$. 
Let $\l_1$ be the first Dirichlet eigenvalue of the Laplacian on $M$. Then
\begin{align*}
   & \l_{1} \geq \bar\l_1(n,\k,\Lambda,R) , 
\end{align*}
where $\bar\l_1(n,\k,\Lambda,R)$
is the first eigenvalue of the one-dimensional eigenvalue problem 
\begin{equation}\label{eq 1.4}
    \begin{cases}
   \vp'' -(n-1)T_{\kappa, \Lambda} \vp' =-\l\vp, \\
    \vp(0)=0,     \vp'(R)=0.
    \end{cases}
\end{equation}
\end{thm}

In \cite{LW21TAMS}, the authors of the present paper obtained an analogous theorem in the K\"ahler setting.
\begin{thm}\label{Thm C}
Let $(M^m,g,J)$ be a compact K\"ahler manifold with smooth nonempty boundary $\p M$. 
Suppose that the holomorphic sectional curvature is bounded from below by $4\k_1$ and the orthogonal Ricci curvature is bounded from below by $2(m-1)\k_2$ for some $\k_1,\k_2 \in \R$, and the second fundamental form on $\p M$ is bounded from below by $\Lambda \in \R$. Let $\l_1$ be the first Dirichlet eigenvalue of the Laplacian on $M$. Then 
$$\l_1\ge\bar{\l}_1(m, \k_1,\k_2, \Lambda, R)$$
where $\bar{\l}_1(m, \k_1,\k_2, \Lambda, R)$ is the first eigenvalue of the one-dimensional eigenvalue problem
\begin{equation}\label{}
    \begin{cases}
   \vp''-\left(2(m-1)T_{\k_2, \Lambda}+T_{4\k_1, \Lambda} \right)\vp'  =-\l\vp, \\
    \vp(0)=0,     \vp'(R)=0.
    \end{cases}
\end{equation}
\end{thm}
\begin{remark}
Diameter-independent lower bounds for $\l_1$  were obtained by Guedj, Kolev and Yeganefar \cite{GKY13}. It has been generalized via a $p$-Reilly formula to the first Dirichlet eigenvalue of the $p$-Laplacian by Blacker and Seto \cite{BS19} when $p \geq 2$. 
\end{remark}

Here we prove the following quaternion-K\"ahler version. 

\begin{thm}\label{Thm Dirichlet}
Let $(M^m,g,I,J,K)$ be a compact quaternion-K\"ahler manifold with smooth nonempty boundary $\p M$. 
Suppose that the scalar curvature is bounded from below by $16m(m+2)\k$ for some $\k \in \R$, and the second fundamental form on $\p M$ is bounded from below by $\Lambda \in \R$. Let $\l_1$ be the first Dirichlet eigenvalue of the Laplacian on $M$. Then 
$$\l_1\ge\bar{\l}_1(m, \k, \Lambda, R)$$
where $\bar{\l}_1(m, \k, \Lambda, R)$ is the first eigenvalue of the one-dimensional eigenvalue problem
\begin{equation}\label{eq 1.5}
    \begin{cases}
   \vp''-\left(4(m-1)T_{\k, \Lambda}+3T_{4\k, \Lambda} \right)\vp'  =-\l\vp, \\
    \vp(0)=0,     \vp'(R)=0.
    \end{cases}
\end{equation}
\end{thm}

\begin{remark}
When the boundary is convex, namely $\Lambda=0$, it is easily seen that we have 
\begin{equation*}
  \bar{\l}_1(m,\k,0,R) =\bar{\mu}_1(m,\k,R).
\end{equation*}
Thus we can obtain the same explicit lower bounds as in Proposition \ref{prop 1.2} for $\bar{\l}_1(m,\k,0,R)$. 
\end{remark}
\begin{remark}
With the help of a generalized Barta's inequality for the $p$-Laplacian (see \cite[Theorem 3.1]{LW20a}), the same argument here indeed yields such lower bounds for the first Dirichlet eigenvalue of the $p$-Laplacian for all $1<p<\infty$. 
\end{remark}

The proof of Theorem \ref{Thm Dirichlet} relies on a comparison theorem for the second derivatives of $d(x,\p M)$.

\begin{thm}\label{Thm comparison distance to boundary}
Let $(M^m,g,I,J,K)$ be a compact quaternion-K\"ahler manifold with smooth nonempty boundary $\p M$. 
Suppose that  the scalar curvature of $M$ is bounded from below by $16m(m+2)\k$ for some $\k \in \R$, and the second fundamental form on $\p M$ is bounded from below by $\Lambda \in \R$. 
Let $\vp:[0, R ] \to \R_+$ be a smooth function with $\vp' \geq 0$. Then the function $v(x) =\vp\left(d(x,\p M)\right)$ is a viscosity supersolution of 
\begin{equation*}
    Q[v] =\left. \left[\a (\vp')\vp'' -\b(\vp')\vp'\left(4(m-1)T_{\k, \Lambda}+3T_{4\k, \Lambda} \right) \right] \right|_{d(x,\p M)},
\end{equation*}
on $M$, where $Q$ is the operator defined in \eqref{eq isotropic}.
\end{thm}
\begin{proof}[Proof of Theorem \ref{Thm comparison distance to boundary}]
By approximation, it suffices to consider the case $\vp' >0$ on $[0,  R]$. 
By definition of viscosity solutions (see \cite{CIL92}), it suffices to prove that for any smooth function $\psi$ touching $v$ from below at $x_0 \in M$, i.e., 
\begin{align*}
    \psi(x) \leq v(x) \text{ on } M, \text{\quad}
    \psi(x_0) &= v(x_0),
\end{align*}
it holds that
\begin{equation*}
    Q[\psi](x_0) \leq \left. \left[\a (\vp')\vp'' -\b(\vp')\vp'\left(4(m-1)T_{\k, \Lambda}+3T_{4\k, \Lambda} \right) \right] \right|_{d(x_0,\p M)}.
\end{equation*}
Since the function $d(x, \p M)$ may not be smooth at $x_0$, so we need to replace it by a smooth function $\bar{d}(x)$ defined in a neighborhood $U(x_0)$ of $x_0$ satisfying $\bar{d}(x) \geq d(x, \p M)$ for $x \in U(x_0)$ and $\bar{d}(x_0)=d(x_0, \p M)$. 
The construction is standard (see e.g. \cite[pp. 73-74]{Wu79} or \cite[pp. 1187]{AX19}), which we state below for reader's convenience. 

Since $M$ is compact, there exists $y_0 \in \p M$ such that $$d(x_0,y_0)=d(x_0, \p M):=s_0.$$ Let $\gamma:[0,s_0] \to M$ be the unit speed length-minimizing geodesic with $\gamma(0)=x_0$ and $\gamma(s_0)=y_0$.  
Choose an orthonormal basis $\{e_i\}_{i=1}^{4m}$ for $T_{x_0}M$ with
\begin{equation*}
    e_1=\gamma_0'(0), e_2=I\gamma_0'(0), e_3=J\gamma_0'(0), e_4=K\gamma_0'(0),
\end{equation*}
and parallel transport it along $\gamma_0$ to produce an orthonormal basis $\{e_i(s)\}_{i=1}^{4m}$ for $T_{\gamma_0(s)}M$ with $e_1(s)=\gamma_0'(s)$ for each $s\in [0,s_0]$.

For any vector $X \in \exp^{-1}_{x_0}U(x_0)$, let $X(s), s\in [0,s_0]$ be the vector field obtained by parallel translating $X$ along $\gamma$, and decompose it as 
\begin{equation*}
    X(s) =a X_1 (s) +b \gamma'(s)+cX_2 (s) ,
\end{equation*}
where $a, b$ and $c$ are constants along $\gamma$ with $a^2+b^2+c^2 =|X|^2$, and $X_1(s)$ and $X_2(s)$ are unit components in $\operatorname{span}\{e_5(s),
\cdots,e_{4m}(s)\}$ and 
$\operatorname{span}\{e_2(s),
e_3(s),e_4(s)\}$. Define 
\begin{equation*}
    W(s)=a \, \eta(s) X_1(s) + b\left(1-\frac{s}{s_0} \right)\gamma'(s)+c\, \zeta(s)X_2(s),
\end{equation*}
where $\eta, \zeta:[0,s_0] \to \R_+$ are two $C^2$ functions to be chosen later. 
Next we define the $4m$-parameter family of curves $\gamma_X :[0,s_0] \to M$ such that 
\begin{enumerate}
    \item $\gamma_0 =\gamma$;
    \item $\gamma_X(0) =\exp_{x_0}(W(0))$ and $\gamma_X(s_0) \in \p M$;
    \item $W(s)$ is induced by the one-parameter family of curves $\varepsilon \to \gamma_{\varepsilon X}(s)$ for $\varepsilon \in (-\varepsilon_0,\varepsilon_0)$ and $s\in [0,s_0]$;
    \item $\gamma_X$ depends smoothly on $X$.
\end{enumerate}
Finally let $\bar{d}(x)$ be the length of the curve $\gamma_X$ where $x=\exp_{x_0}(X) \in U(x_0)$. 
Then we have
$\bar{d}(x) \geq d(x,\p M)$ on $U(x_0)$, $\bar{d}(x_0)=d(x_0, \p M)$. 

Recall the  first and second variation formulas:
\begin{equation}\label{1st bard}
\n \bar{d} =-e_1(0),
\end{equation}
and 
\begin{eqnarray*}
   && \n^2 \bar{d} (X,X)\\
    &=& -A\Big(a \eta (s_0)X_1(s_0)+c\zeta(s_0)X_2(s_0),a \eta (s_0)X_1(s_0)+c\zeta(s_0)X_2(s_0)\Big) \\
    && + \int_0^{s_0} \left(a^2(\eta')^2+c^2(\zeta')^2 -R(a \eta X_1+c\zeta X_2, \gamma',a \eta X_1+c\zeta X_2, \gamma') \right) ds
\end{eqnarray*}
where $A$ denotes the second fundamental form of $\p M$ at $y_0$. Here and below the derivatives of $\bar{d}$ are all evaluated at $x_0$.
 Then we have 
\begin{equation}\label{2nd bardn}
\n^2 \bar{d} (e_1(0),e_1(0))=0.
\end{equation}
For $ 2 \le i\le 4$, we obtain by choosing $\zeta(s)=C_{4\kappa, \Lambda}(s_0-s) /C_{4\kappa, \Lambda}(s_0)$  that 
\begin{eqnarray}
    && \n^2 \bar{d} (e_i(0),e_i(0)) \\ \nonumber 
    &=& - \zeta^2 (s_0) A(e_i(s_0), e_i(s_0)) + \int_0^{s_0} (\zeta')^2 -\zeta^2 R( e_i, \gamma', e_i,\gamma') \, ds \\ \nonumber
    &\le & -\frac{\Lambda}{C_{4\k, \Lambda}^2 (s_0)} +\int_0^{s_0} (\zeta')^2 -\zeta^2 R( e_i, \gamma', e_i,\gamma') \, ds\nonumber
\end{eqnarray}
Summing over $2\leq i \leq 4$  gives 
\begin{eqnarray}
     && \sum_{i=2}^{4} \n^2 \bar{d} (e_i(0),e_i(0)) \\ \nonumber 
     &\le & - \frac{3\Lambda}{C_{4\k, \Lambda}^2 (s_0)} + \int_0^{s_0} 3(\zeta')^2 -\zeta^2 \sum_{i=2}^4R( e_i, \gamma', e_i,\gamma') \, ds \\\nonumber
      &\le & - \frac{3\Lambda}{C_{4\k, \Lambda}^2 (s_0)} + \int_0^{s_0} 3(\zeta')^2 -12\k \zeta^2  \, ds \\\nonumber
    &=&  -3T_{4\k, \Lambda}(s_0). \nonumber
\end{eqnarray}
For $5\le i \le 4m$, we have 
\begin{eqnarray*}
    \n^2 \bar{d} (e_i(0),e_i(0)) 
    =- \eta^2 (s_0) A(e_i(s_0), e_i(s_0)) + \int_0^{s_0} (\eta')^2 -\eta^2 R(e_i, \gamma',e_i, \gamma') \, ds.
\end{eqnarray*}
Summing over $5\leq i \leq 4m$ and choosing $\eta(s)=C_{\k, \Lambda}(s_0-s) /C_{\k, \Lambda}(s_0)$  gives 
\begin{eqnarray}
     && \sum_{i=5}^{4m} \n^2 \bar{d} (e_i(0),e_i(0)) \\ \nonumber 
     &=& - \frac{\sum_{i=5}^{4m}A(e_i(s_0), e_i(s_0))}{C_{\k, \Lambda}^2 (s_0)} +    \int_0^{s_0} 4(m-1)(\eta')^2  -\eta^2 \sum_{i=5}^{4m}R( e_i, \gamma', e_i,\gamma') \, ds  \\\nonumber
    &\leq& - \frac{4(m-1)\Lambda}{C_{\k, \Lambda}^2 (s_0)} +4(m-1)\int_0^{s_0} (\eta')^2 -\k \eta^2  \, ds \\\nonumber
    &=& -4(m-1)T_{\k, \Lambda}(s_0). \nonumber
\end{eqnarray}
Since the function $\psi(x) -\vp\left(d(x, \p M)\right)$ attains its maximum at $x_0$ and $\vp' >0$, it follows that the function $\psi(x) -\vp(\bar{d}(x))$ attains a local maximum at $x_0$. The first and second derivative tests yield
$$ \n  \psi (x_0) =-\vp' e_1 (0), \quad \psi_{11}(x_0)  \leq  \vp'',$$
    and
    $$
    \psi_{ii}(x_0)  \leq \vp' \n^2\bar{d} \left(e_i(0), e_i(0)\right)
$$
for $2 \le i \leq 4m$, where we used (\ref{1st bard}) and (\ref{2nd bardn}).
Here and below the derivatives of $\vp$ are all evaluated at $s_0=d(x_0, \p M)$. 
Thus we have 
\begin{eqnarray}\label{eq 3.1} 
     Q[\psi](x_0) &=& \a (\vp')\psi_{11} +\b(\vp') \sum_{i=2}^{4m} \psi_{ii} \\\nonumber
    &\leq& \a (\vp')\vp'' +\b(\vp')\vp'\left(\sum_{i=2}^{4m}  \n^2 \bar{d}(e_i(0),e_i(0)) \right) \\
    &\leq & \a (\vp')\vp'' -\b(\vp')\vp'\big(4(m-1)T_{\k, \Lambda}+3T_{4\k, \Lambda} \big).\nonumber
\end{eqnarray}
The proof is complete. 
\end{proof}
By choose $\vp(s)=s$, we have
\begin{corollary}\label{Cor 6.2}
Let $(M^m,g,I,J,K)$ be a compact quaternion-K\"ahler manifold with smooth nonempty boundary $\p M$. 
Suppose that  the scalar curvature of $M$ is bounded from below by $16m(m+2)\k$ for some $\k \in \R$, and the second fundamental form on $\p M$ is bounded from below by $\Lambda \in \R$. 
Then 
\begin{equation*}
  \Delta d(x,\p M)\le  -4(m-1)T_{\k, \Lambda}(d(x,\p M))-3T_{4\k, \Lambda} (d(x,\p M))
\end{equation*}
in the viscosity sense.
\end{corollary}

\begin{proof}[Proof of Theorem \ref{Thm Dirichlet}]
The proof is a slight modification of the K\"ahler case presented in \cite[Section 7]{LW21TAMS}, so we omit the details. 
\end{proof}

\bibliographystyle{alpha}
\bibliography{ref}

\end{document}